\documentclass[onenumber, plainnumbering,english,11pt]{amsart}
\usepackage[square,numbers]{natbib}
\usepackage{amssymb,amsmath,amscd,mathrsfs,paralist}
\usepackage[mathscr]{eucal}
\usepackage[utf8]{inputenc}
\usepackage[T1]{fontenc}
\usepackage{bm}
\usepackage{graphicx}
\usepackage{tikz}
\usepackage{pgfplots}
\usetikzlibrary{positioning, shapes.geometric}
\usepgfplotslibrary{colormaps}
 
\pgfplotsset{
    colormap={mycolormap}{color=(lightgray) color=(white) color=(lightgray)}
}
\usepackage{tikz-cd}
\usetikzlibrary{arrows, matrix}
\usepackage[all,cmtip]{xy}
\usepackage[colorlinks,linkcolor=black,urlcolor=black]{hyperref}
\hypersetup{
     colorlinks   = true,
     citecolor    = black}
\usepackage{pgfplots}
\usepackage[font=small,labelsep=none]{caption}
\newcommand{\R}{\mathbb R}
\newcommand{\N}{\mathbb N}
\newcommand{\Z}{\mathbb Z}
\newcommand{\C}{\mathbb C}

\newcommand{\lesss}{\rotatebox[origin=c]{90}{$\land$}}
\newcommand{\less}{\ \lesss\ }
\newcommand{\biggg}{\rotatebox[origin=c]{90}{$\lor$}}
\newcommand{\bg}{\ \biggg\ }
\newcommand{\nc}{\newcommand}
\nc{\BCc}{{\mathbb{C}(\wp(z),\wp^\prime(z))}}
\nc{\BC}{{\mathbb C}}
\nc{\BQ}{{\mathbb Q}}
\nc{\BR}{{\mathbb R}}
\nc{\BZ}{{\mathbb Z}}
\nc{\BP}{{\mathbb P}}
\nc{\BN}{{\mathbb N}}
\nc{\BM}{{\mathbb M}}
\nc{\fH}{{\mathfrak{H}}}
\nc{\vp}{{\varepsilon}}\nc{\dpar}{{\partial}}\nc{\al}{{\alpha}}
\nc{\PSL}{PSL(2,\BR)}
\nc{\PS}{PSL(2,\BZ)}
 \nc{\CL}{PSL(2,\BZ/m\BZ)}
 \newtheorem{theorem}{Theorem}[section]

\newtheorem{lemma}[theorem]{Lemma}
\newtheorem{proposition}[theorem]{Proposition}
\theoremstyle{definition}
\newtheorem{definition}[theorem]{Definition}
\newtheorem{remark}[theorem]{Remark}
\newtheorem{example}[theorem]{Example}

\newcommand{\theoref}[1]{Theorem~\ref{#1}}

\newcommand{\lemref}[1]{Lemma~\ref{#1}}

\newcommand{\secref}[1]{Section~\ref{#1}}

\DeclareMathOperator{\rad}{\mathrm{rad}}
\DeclareMathOperator{\Hom}{\mathrm{Hom}}
\DeclareMathOperator{\Id}{\mathrm Id}

\DeclareMathOperator{\spec}{\mathrm Spec}
\DeclareMathOperator{\proj}{\mathrm Proj}
\newenvironment{prooof}
	{\textit{\textbf{Proof of Theorem 3.3.}}}
	{\hfill $\square$\vskip 8pt}
\newenvironment{proooof}
	{\textit{\textbf{Proof of Theorem 3.7.}}}
	{\hfill $\square$\vskip 8pt}
	 \newenvironment{rcases}
  {\left.\begin{aligned}}
  {\end{aligned}\right\rbrace}

\numberwithin{equation}{section}  
\begin{document}
\title{Representation theory and differential equations}
\address{Ahmed Sebbar\\
Chapman University, One University Drive, 
   \\92866, Orange, CA \&
   Univ. Bordeaux, IMB 
   \\UMR 5251, F-33405 Talence, France.}
   
         \email{sebbar@chapman.edu, ahmed.sebbar@math.u-bordeaux.fr.}
\author{Ahmed Sebbar and Oumar Wone}
\address{Oumar Wone\\
}
         \email{wone@chapman.edu}
\date{}
\begin{abstract}
We study the geometry and partial differential equations arising from the consideration of group-determinants, and representation theory. The simplest and most striking such example is undoubtedly that of the Humbert operator, associated with the cyclic group $\Z/3\Z$, 
$\displaystyle \Delta_3=\dfrac{\partial^3}{\partial x^3}+\dfrac{\partial^3}{\partial y^3}+\dfrac{\partial^3}{\partial z^3}-3\dfrac{\partial^3}{\partial x\partial y\partial z}$.  This operator appears as a natural extension of the Laplacian in dimension 2. Another originality of our work is to show that the spectral theory of operators associated with Frobenius determinants is closely linked to finite Fourier transform theory.

\end{abstract}
\thanks{Acknowledgement: The authors would like to thank the referees for their careful reading and for the reference \cite{krasnov}. We plan, in a subsequent study, to investigate the relations of \cite{krasnov} with our work. The second author expresses warm thanks for a generous invitation to Chapman University, where this work was initiated, and for the creative and friendly atmosphere which he experienced during the stay}

\maketitle 
\tableofcontents
\section{Introduction}
\label{intro}
The solution of a partial differential equation may often be obtained in a very convenient form by the consideration of definite integrals: thus for instance, the solution of the heat equation $\displaystyle\dfrac{\partial^2V}{\partial x^2}=\dfrac{\partial V}{\partial t}$ which reduces to $f(x)$ when $t=0$ is given by
\begin{equation}
\label{f1}
\displaystyle V=\dfrac{1}{\sqrt{\pi}}\int_{-\infty}^{+\infty}f(x+2u\sqrt{t})e^{-u^2}du.
\end{equation}

Whittaker in \cite{whittaker1} gave the complete solution of the three dimensional Laplace's equation in the form of a definite integral. 
It is Radon who started in 1917, the study of integral transforms. He studied transforms of functions $f:\R^2\to\R$ with appropriate decay conditions at infinity. More precisely to $L\subset \R^2$, an oriented line, he associates
 $$\displaystyle\phi(L):=\int_Lf.$$
 Remark that there exists an inversion formula for the Radon transform, see \cite{gelfand2000, radon1917}. The next important study of integral transforms is due to F. John in $1938$ \cite{gelfand2000, john1938}. He was interested in integral transforms of functions $f:\R^3\to\R$ (satisfying appropriate conditions which make the integrals well-defined). He defined, for $L$ an oriented line, the transform
 $$\displaystyle\phi(L):=\int_Lf$$
 that is
 $$\displaystyle\phi(\alpha_1,\alpha_2,\beta_1,\beta_2)=\int_{-\infty}^{+\infty}f(\alpha_1s+\beta_1,\alpha_2s+\beta_2,s)ds.$$
 Because the space of oriented lines in $\R^3$ is four dimensional, it is expected that $\phi$ has to satisfy $1=4-3$ extra condition. To see this we differentiate under the integral sign to obtain the ultrahyperbolic wave equation
 $$\displaystyle\dfrac{\partial^2\phi}{\partial\alpha_1\partial\beta_2}=\dfrac{\partial^2\phi}{\partial\alpha_2\partial\beta_1}.$$
 After changing the coordinates $\alpha_1=x+y,\,\alpha_2=t+z,\,\beta_1=t-z,\,\beta_2=x-y$, the ultrahyperbolic wave equation becomes
 $$\displaystyle\dfrac{\partial^2\phi}{\partial x^2}+\dfrac{\partial^2\phi}{\partial z^2}-\dfrac{\partial^2\phi}{\partial y^2}-\dfrac{\partial^2\phi}{\partial t^2}=0.$$ 
 Later on in 1967, R. Penrose \cite{penrose1967} introduced the twistor transform in order to find solutions of the wave equation in Minkowski space. This was the first time sophisticated methods like cohomology, were used in such questions.
 
Therefore the problem of giving integral representations of solutions to constant coefficients differential equations is an interesting problem. In this direction one can remind the famous Ehrenpreis-Malgrange-Palamodov theorem (see \cite{bjork}) beside the above cited work of Penrose, F. John, Whittaker and the work of Bateman \cite{bateman1}.

Recently \cite{murray1985, eastwood19851}, the results given by Whittaker have been generalized (by using cohomological methods, but with no integral representation) for the partial differential equation
$$\displaystyle  D_f\phi:=f\left(\dfrac{\partial}{\partial x_0},\ldots,\dfrac{\partial}{\partial x_n}\right)\phi(x_0,\ldots,x_n)=0$$
arising from the consideration of a polynomial $f\left( x_0,\ldots, x_n\right),\,n\geq 2$. In the paper \cite{eastwood19851}, a full description of the holomorphic functions in the kernel of $D_f$ is given by using an approach based on the twistor transform. 

One of the themes of this paper is the study of the partial differential operators arising from Frobenius determinants. Let $G$ be a finite group of order $\sharp{G}=n$ and $X_g$ a set of independent variables indexed by the elements of the group $G$. Then we define the Frobenius matrix as $M=(X_{gh^{-1}})_{g,h}$ and the Frobenius determinant as 
\begin{equation*}
\begin{split}
\label{f388}
\Theta(G)((X_g)_{g\in G})=\det(X_{gh^{-1}}).
\end{split}
\end{equation*}
This determinant is a homogeneous polynomial in the variables $X_g$ of degree $n$. Moreover it has important factorization properties, see \theoref{dedekind} and \theoref{frobenius1}. If we replace the variables $X_g$ by the corresponding partial differential symbols $\partial_g=\frac{\partial}{\partial X_g}$, $\Theta(G)$ becomes a partial differential operator $\Theta(G)((\partial_g)_{g\in G})$ in the symbols $\partial_g$ and the factorization property of $\Theta(G)$ gives a factorization of $\Theta(G)((\partial_g)_{g\in G})$. This enables us, see \secref{frobenius}, to find holomorphic functions in the kernel of $\Theta(G)((\partial_g)_{g\in G})$. 

According to the fundamental theorem of finite Abelian groups, any finite Abelian group is isomorphic to a direct product of cyclic groups of prime-power order, where the decomposition is unique up to the order in which the factors are written. Consequently the spectral theory of operators associated with Group determinants is closely linked to finite Fourier transform theory and circulant matrices. This  is illustrated below in the examples of $\displaystyle \Z/4\Z $ and $\Z/2\Z \times \Z/2\Z $.

In \secref{frobenius} we study the Eigenvalue problem of operators arising from Frobenius determinants in the case of abelian groups. 

Finally in (our principal) \secref{unitsphere} we study the structure of the linear algebraic group consisting of Frobenius matrices which are invertible, for a fixed group $G$. We also establish links between Frobenius determinants and almost-Frobenius structures as well as toric geometry. The appearance of the structures of certain toric varieties is due precisely to the fact that the symbols of Group determinants partial differential operators  for finite abelian groups, factorize into linear polynomials.

For an excellent review on group-determinant we refer to \cite{conrad1998}. We are aware that some of our proofs could be shortened by using more powerful techniques in algebra or in differential or algebraic geometry but we restrict to the methods used below for simplicity.

In this paper, our study is of analytic character, but we should point out that an important application of Frobenius determinants is made in algebraic number theory \cite[chap.~21]{lang1987} where one of the striking results says that if $K$ is an imaginary quadratic field and $H$ is the Hilbert class field of $K$, the Dedekind zeta function $\zeta_H(s)$ can be expressed as the Frobenius determinant of a matrix whose entries are the partial zeta functions of $K$.
\section{Partial differential equations defined by group-determinants}
\label{frobenius}
\subsection{"The kernels" of the PDEs arising from Frobenius determinants}
Let $G$ be a finite group of order $\sharp(G)=n$ and $X_g$ a set of independent variables indexed by the elements of the group $G$. We define the group determinant or Frobenius determinant as
$$\Theta(G)((X_g)_{g\in G})=\det(X_{gh^{-1}}).$$ 
In the abelian case we have
\begin{theorem}[Dedekind, Frobenius, \cite{conrad1998}]
\label{dedekind}
For $G$ a finite abelian group one has
$$\Theta(G)((X_g)_{g\in G}):=\det(X_{gh^{-1}})=\displaystyle\prod_{\chi\in\widehat{G}}(\sum_{g\in G}\chi(g)X_g)$$ 
where $\widehat{G}$ is the character group of $G$, i.e. the group of homomorphisms from $G\to \C^\times$.
\end{theorem}
\begin{example}
\label{ex}
Let $G=\Z/n\Z$ then the group determinant reduces to the circulant determinant defined as
$$\Theta(\Z/n\Z)(X_0,\ldots,X_n)=\displaystyle \left|\begin{array}{ccccc}X_0 & X_1 & X_2 & \ldots & X_{n-1} \\X_{n-1} & X_0 & X_1 & \ldots & X_{n-2} \\\vdots & \vdots & \vdots & \vdots & \vdots \\X_1 & X_2 & X_3 & \ldots & X_0\end{array}\right|=\prod_{\omega\in\C,\,\omega^n=1}(\sum_{0}^{n-1}\omega^{k}X_k),$$
with $\omega$ a primitive $n$-th root of unity.

Let us make precise the calculation of the possible group determinants when $n=4$. There two finite (abelian) groups of order $4$, $\Z/4\Z$ and the Klein group $\Z/2\Z\times\Z/2\Z$. We first start with the case of the cyclic group $\Z/4\Z=\{0,1,2,3\}$. Its character table is
$$\begin{tabular}{|c|c|c|c|c|}\hline  $\Z/4\Z$& 0 & 1 & 2 & 3 \\\hline $\chi_0$ & 1 & 1 & 1 & 1 \\\hline $\chi_1$ & 1 & i & -1 & -i \\\hline $\chi_2$ & 1 & -1 & 1 & -1 \\\hline $\chi_3$ & 1 & -i & -1 & i \\\hline \end{tabular}$$
This gives
\begin{equation*}
\begin{split}
\Theta(\Z/4\Z)(X_0,X_1,X_2,X_3)&=(X_0+X_1+X_2+X_3)(X_0-X_1+X_2-X_3)\\&(X_0+iX_1-X_2-iX_3)
(X_0-iX_1-X_2+iX_3)\\
&=\{(X_0+X_2)^2-(X_1+X_3)^2\}\{(X_0-X_2)^2+(X_1-X_3)^2\}.
\end{split}
\end{equation*}
For the case of the Klein group $\Z/2\Z\times\Z/2\Z=\{(0,0),(1,0),(0,1),(1,1)\}$ we have the following character table
$$\begin{tabular}{|c|c|c|c|c|}\hline $\Z/2\Z\times\Z/2\Z$ & (0,0) & (1,0) & (0,1) & (1,1) \\\hline $\chi_{(0,0)}$ & 1 & 1 & 1 & 1 \\\hline $\chi_{(1,0)}$ & 1 & -1 & 1 & -1 \\\hline $\chi_{(0,1)}$ & 1& 1 & -1 & -1 \\\hline $\chi_{(1,1)}$ & 1 & -1 & -1 & 1 \\\hline \end{tabular}.$$
This gives
\begin{equation*}
\begin{split}
\Theta(\Z/2\Z\times\Z/2\Z)(X_0,X_1,X_2,X_3)&=(X_0+X_1+X_2+X_3)(X_0+X_1-X_2-X_3)\\&(X_0-X_1+X_2-X_3)
(X_0-X_1-X_2+X_3)\end{split}
\end{equation*}
where we identify $(0,0)$ with $0$, $(1,0)$ with $1$, $(0,1)$ with $2$ and $(1,1)$ with $3$.
\end{example}

If in the example above we replace the independent variables $X_i$ by $\partial_i:=\dfrac{\partial}{\partial x_i}$, $i=0,\ldots,n-1$. We obtain the differential operator
\begin{equation}
\label{f43}
\Theta(\Z/n\Z)(\partial_0,\ldots,\partial_{n-1})=\prod_{\omega\in\C,\,\omega^n=1}(\sum_{0}^{n-1}\omega^k\partial_k).
\end{equation}
The case $n=3$ gives the Humbert Laplacian \cite{humbert1929}
$$\Delta_3=\dfrac{\partial^3}{\partial x^3}+\dfrac{\partial^3}{\partial y^3}+\dfrac{\partial^3}{\partial z^3}-3\dfrac{\partial^3}{\partial x\partial y\partial z}.$$
Let $\mathscr{V}$ be the Vandermonde matrix constructed from the $n$-th roots of unity $\omega_0,\ldots,\omega_{n-1}$ and let
\begin{equation}
\label{f44}
^t(u_{\omega_0},\ldots,u_{\omega_{n-1}}):=\mathscr{V}^{-1}\, ^t(X_0,\ldots,X_{n-1}).
\end{equation}
Then because of the non-vanishing of the Vandermonde determinant constructed with the roots of unity $\omega_i$, the $u_\omega$ are independent variables. We now consider the equation
\begin{equation}
\label{f45}
\Theta(\Z/n\Z)(\partial_0,\ldots,\partial_{n-1})w=0
\end{equation}
 with integer coefficients. From the factorization property of the differential operator $\Theta(\Z/n\Z)(\partial_0,\ldots,\partial_{n-1})$, we have by making use of the change of variables
 $$^t(u_{\omega_0},\ldots,u_{\omega_{n-1}})$$
 \begin{equation}
 \label{f46}
\displaystyle \Theta(\Z/n\Z)(\partial_0,\ldots,\partial_{n-1})w=0\iff\prod_{i=0}^{n-1}\partial_{u_{\omega_i}}w=0. 
 \end{equation}
 Let $\omega_0,\omega_1,\ldots,\omega_{n-1}$ be the $n$ roots of $\omega^n=1$. The general solution to the partial differential equation $\Theta(\Z/n\Z)(\partial_0,\ldots,\partial_{n-1})w=0$ is given by
 \begin{equation}
 \label{f47}
 \sum_{0}^{n-1}g_{\omega_i}(u_{\omega_0},u_{\omega_1},\ldots,\widehat{u_{\omega_i}},\ldots,u_{\omega_{n-1}}),
 \end{equation}
with $g_{\omega_i}$ is holomorphic for every $\omega_i$ and $\widehat{u_{\omega_i}}$ means that we omit the corresponding variable.

  \begin{lemma}
 \label{devisme}
 Let $(\alpha_0,\ldots,\alpha_{n-1})\in\C^n\setminus\{0\}$ such that $\Theta(\Z/n\Z)(\alpha_0,\ldots,\alpha_{n-1})=0$ and $f$ holomorphic in $\C$. Then $f(\alpha_0x_0+\ldots+\alpha_{n-1}x_{n-1})$ is annihilated by $\Theta(\Z/n\Z)(\partial_0,\ldots,\partial_{n-1})$.
 \end{lemma}
 \begin{proof}
 This follows from the chain rule and the definition of $\Theta(\Z/n\Z)(\partial_0,\ldots,\partial_{n-1})$.
 \end{proof}
In general
 \begin{lemma}
 \label{f51}
 Let $G$ be a finite abelian group and $\widehat{G}$ its set of characters. Then the determinant of the matrix
 $$M=(\chi(g)),\,(\chi,g)\in (\widehat{G}\times G)$$
 is non zero. 
 \end{lemma}
 \begin{proof}
 This is classical and follows from the consideration of the group algebra $\C[G]$ which admits the two bases $(g)_{g\in G}$ and $\displaystyle(\sum_{g\in G}\chi(g)g)_{\chi\in \widehat{G}}$.
 \end{proof}
To a given $X_g,\,g\in G$ we associate the partial derivative $\partial_g=\dfrac{\partial}{\partial X_g}$ and we obtain the associated partial differential equation
  \begin{equation}
  \label{f52}
  \Theta(G)((\partial_g)_{g\in G})w=0.
  \end{equation} 
   \lemref{f51} allows us to introduce the following independent variables $u_\chi$
  $$\sum_{\chi\in \widehat{G}}\chi(g)u_{\chi}:=X_g,\,g\in G.$$
  Hence equation \eqref{f52} is equivalent to
  \begin{equation}
  \label{f53}
  \displaystyle \prod_{\chi\in\widehat{G}}\partial_{u_{\chi}}w=0,
  \end{equation}
  which is a separation of variable of \eqref{f52}. Likewise the \lemref{devisme} immediately generalizes to the case of an arbitrary finite abelian group $G$.
  
  We consider the general case of the partial differential equation defined by the Frobenius determinant when the finite group $G$ is non-abelian (hence its order is at least $\geq 6$). That is we consider the equation
   \begin{equation}
   \label{f57}
   \Theta(G)((\partial_g)_{g\in G})w=0,
   \end{equation}
   where the notation $\partial_g$ is defined analogously as in equation \eqref{f52}. First of all similarly to \theoref{dedekind} we have \cite{conrad1998, dickson1902}
   \begin{theorem}[Frobenius]
   \label{frobenius1}
   Let $G$ be a finite non-abelian group of order $n\geq 6$. Let $s$ be the number of its conjugacy classes. Then the group determinant $\Theta(G)((X_g)_{g\in G})$ as a polynomial in the variables $(X_g)_{g\in G}$ admits the following factorization
   $$\displaystyle \Theta(G)((X_g)_{g\in G})=\prod_{i=1}^s\Phi_i^{f_i}((X_g)_{g\in G}).$$
   In this expression each $\Phi_i$ is an irreducible polynomial in the variables $(X_g)_{g\in G}$ of degree $f_i$ with $f_i|n$ and $\sum_{i=1}^sf_i^2=n$.
   \end{theorem}
  Due to the independence of the variables $(X_g)_{g\in G}$, the partial derivatives $(\partial_g)_{g\in G}$ commute. Therefore the evaluations of the $\Phi_i$ on these partial derivatives also commute. Thus to give a solution of the equation \eqref{f57}, it suffices to find a non-trivial function in the kernel of any of the operators
   $$\Phi_i((\partial_g)_{g\in G}).$$
   A fundamental property of each factor $\Phi_i$ of the group determinant is that one can write it in good coordinates as the determinant of a square $f_i\times f_i$ matrix the entries of which are linear polynomials in $(X_g)_{g\in G}$ and which are algebraically independent over $\C$, \cite[p.~377]{conrad1998}. Therefore one can write $\Phi_i$ as $\det((z_{jl})_{1\leqslant j\leqslant f_i,1\leqslant l\leqslant f_i})$ with $(z_{jl})_{1\leqslant j\leqslant f_i,1\leqslant l\leqslant f_i}$ a matrix of variables. So in order to find a holomorphic function in the kernel of $\Phi_i(\partial_g)$ it suffices to find one in the kernel of $\Omega=\det(\partial_{z_{jl}})$. A central observation is that $\Omega$ is the Cayley operator of classical invariant theory, see \cite{dolgachev2003, mukai2003, olver1999}. One of its important properties is that for any $A\in M_{f_i}(\C) $ and any polynomial $q$ in $Z=(z_{jl})_{1\leqslant j\leqslant f_i,1\leqslant l\leqslant f_i}$ we have
   $$\Omega(A\star q)=\det(A)(A\star \Omega(q)),$$
   here $A\star q(Z):=q(A^tZ)$, with $A^t$ is the transpose of $A$. Taking $A$ to be a matrix of rank strictly less than $f_i$, we see that $\Omega(A\star q)=0$, i.e. $A\star q$ belongs to the kernel of $\Omega$. In a similar spirit one can find functions in the kernel of $\Omega=\det(\partial_{z_{jl}})$ in the following way. Let us first introduce some notations from classical invariant theory. We set
   $$Y=^t(\partial_{z_{jl}})_{1\leqslant j\leqslant f_i,1\leqslant l\leqslant f_i}$$
   and define the operators of polarization $(\Delta_{jl})_{1\leqslant j\leqslant f_i,1\leqslant l\leqslant f_i}$ by the formula \cite[p.~48]{procesi2007}
   $$ZY=:(\Delta_{jl})_{1\leqslant j\leqslant f_i,1\leqslant l\leqslant f_i},$$
   that is
   $$\Delta_{jl}=\displaystyle\sum_{h=1}^{f_i}z_{jh}\frac{\partial}{\partial z_{lh}},\,\,{1\leqslant j\leqslant f_i,1\leqslant l\leqslant f_i}.$$
   For $j\not=l$, $\Omega$ commutes with $\Delta_{jl}$. Therefore for the commutators we have \cite[p.~56]{procesi2007}
   $$\left[\Omega,\Delta_{jl}^r\right]=0,\,\,r\in\Z_{>0}, {1\leqslant j\leqslant f_i,1\leqslant l\leqslant f_i}.$$
   Hence by developing $\Omega=\det(Y)$ with respect to the $l$-th column for $1\leqslant l\leqslant f_i $ one finds that for any holomorphic function $f(z_{l1},z_{l2},\ldots,z_{lf_i})$, $1\leqslant l\leqslant f_i $, the function
   $$\Delta_{jl}^rf(z_{l1},z_{l2},\ldots,z_{lf_i}),\,r\in\Z_{>0}$$
   is in the kernel of $\Omega$.
   
 Moreover if $$\sum_g\alpha_g\frac{\partial}{\partial_{X_g}}, (\alpha_g)_{g\in G}\in\C^n\setminus\{(0,0,\ldots,0)\},\,n=\sharp{G}$$ is a first order differential operator, then every holomorphic function in $\C^n$ in its kernel is of the form
   $$\mathscr F(\alpha_{g_2}X_{g_1}-\alpha_{g_1}X_{g_2},\alpha_{g_3}X_{g_2}-\alpha_{g_2}X_{g_3},\ldots,\alpha_{g_n}X_{g_{n-1}}-\alpha_{g_{n-1}}X_{g_n}).$$
   Here $g_1$, $g_2$, $\ldots$, $g_n$ are the elements of $G$ listed in this order, and $\mathscr F(y_1,y_2,\ldots,y_{n-1})$ is a holomorphic function on $\C^{n-1}$. 
 \begin{proof}
 This follows from \cite[p.~31-32]{goursat1891}.
 \end{proof}  
 Let us now consider the reduced operator 
 $$P((\partial_g)_{g\in G})=\displaystyle\prod_{i=1}^s\Phi_i((\partial_g)_{g\in G})$$
 associated to
 $$\Theta(G)((\partial_g)_{g\in G})=\prod_{i=1}^s\Phi_i^{f_i}((\partial_g)_{g\in G})$$
 in the case of a non-abelian finite group $G$. From what we have just explained above we deduce 

  
  \begin{theorem}
 \label{sol}
 Given a non-abelian finite group $G$ of order $n\geq 6$ with Frobenius determinant $\Theta(G)$ and associated partial differential operator
 $$\Theta(G)((\partial_g)_{g\in G}),$$
 one can find global non-trivial holomorphic functions in its kernel.
 \end{theorem}
 \begin{example}
 \label{john}
 We study the Frobenius determinant of the symmetric group $\mathscr S_3$, the first non-abelian finite group. We relate it to the ultrahyperbolic equation of John and to the classical Gauss hypergeometric function. We enumerate the elements of $\mathfrak S_3$ as follows:
 $$\pi_1=(1),\,\pi_2=(123),\,\pi_3=(132),\,\pi_4=(23),\,\pi_5=(13),\,\pi_6=(12).$$
 Set $X_i=X_{\pi_i}$. Then Dedekind calculated
 $$\Theta(\mathfrak S_3)=\Phi_1\Phi_2\Phi_3^2,$$
 where
\begin{equation*}
\displaystyle
\begin{split}
\Phi_1(X_1,\ldots,X_6)&=X_1+X_2+X_3+X_4+X_5+X_6,\\
\Phi_2(X_1,\ldots,X_6)&=X_1+X_2+X_3-X_4-X_5-X_6,\\
\Phi_3(X_1,\ldots,X_6)&=X_1^2+X_2^2+X_3^2-X_4^2-X_5^2-X_6^2-X_1X_2-X_1X_3-X_2X_3\\
&+X_4X_5+X_4X_6+X_5X_6.
\end{split}
\end{equation*}
If $\omega$ is a primitive cubic root of unity the change of variables is invertible
 \begin{equation*}
 \begin{split}
 u&=X_1+X_2+X_3\\
 v&=X_4+X_5+X_6\\
 u_1&=X_1+\omega X_2+\omega^2X_3\\
 v_1&=X_4+\omega X_5+\omega^2 X_6\\
 u_2&=X_1+\omega^2X_2+\omega X_3\\
 v_2&=X_4+\omega^2X_5+\omega X_6,
 \end{split}
 \end{equation*}
 since its determinant is $\displaystyle \left((1-\omega^2)(\omega-\omega^2)\right)^2\not=0$. We write the factorization of $\Theta(\mathfrak S_3)$ as
 $$\Theta( \mathfrak S_3)(X_1,\cdots,X_6)=(u+v)(u-v)(u_1u_2-v_1v_2)^2.$$
We set $\alpha_1=u_1$, $\beta_2=u_2$, $\beta_1=v_1$ and $\alpha_2=v_2$. Then we have
 $$\Theta(\mathfrak S_3)(X_1,\cdots,X_6)=(u+v)(u-v)(\alpha_1\beta_2-\beta_1\alpha_2)^2.$$
 The associated differential operator is
 $$(\partial_u+\partial_v)(\partial_u-\partial_v)(\partial_{\alpha_1}\partial_{\beta_2}-\partial_{\alpha_2}\partial_{\beta_1})^2,$$
 as previously explained, and $(\partial_u+\partial_v)(\partial_u-\partial_v)$ is the two-dimensional wave operator. We look for generalized solutions. It suffices to know a solution of some factor of the operator. The main point of the discussion is that all solutions of 
 $$\Omega_4=\partial_{\alpha_1}\partial_{\beta_2}-\partial_{\alpha_2}\partial_{\beta_1}=0$$
 in $S(\R^4)$ are known explicitly (\cite[p.~46]{gelfand2000}, \cite{john1938}). They are given by the John's transforms of functions belonging to the Schwartz space $S(\R^3)$, defined by
 \begin{equation}
 \label{f90}
 \displaystyle\psi(\alpha_1,\alpha_2,\beta_1,\beta_2)=\int_{-\infty}^{+\infty}f(\alpha_1x_3+\beta_1,\alpha_2x_3+\beta_2,x_3)dx_3,\,f\in S(\R^3).
 \end{equation}
The Gauss hypergeometric function can be expressed as a John's transform, \cite[p.~46]{gelfand2000}. Indeed one can define the John transform \eqref{f90} whenever the integral involved makes sense. In particular, the John transform $\psi_\lambda$ of the following function $f_\lambda$ on $\R^3$ is well-defined:
 $$f_\lambda(x_1,x_2,x_3)=(x_1)_+^{\lambda_1-1}(x_2)_+^{\lambda_2-1}(x_3)_+^{\lambda_3-1},\quad \lambda=(\lambda_1,\lambda_2,\lambda_3)\in \C^3$$
 where $t^{a}_+=t^{a}$ if $t\bg 0$ and zero for $t\less 0$. For $\Re(a)\bg 0$ the function $t_+^a$ is defined as a classical function and it extends to a distribution in $t$. We have 
 \[\psi_\lambda( \alpha_1,\alpha_2,\beta_1, \beta_2)= \int_{-\infty}^{+\infty}(x_3\alpha_1+ \beta_1)^{\lambda_1-1}_+( x_3\alpha_2+ \beta_2 )^{\lambda_2-1}_+{x_3}_+^{\lambda_3-1}\;dx_3.\]
 This integral converges if $\Re\lambda_i>0,\,i= 1,2,3$, and $\Re(\lambda_1+\lambda_2+\lambda_3)<2$ and in this domain the integral defines an analytic function of $\lambda_1$, $\lambda_2$, $\lambda_3$. Therefore, the function defined by the integral can be extended by the analytic continuation. Restricting to the domain on which $\alpha_i\less 0\less \beta_i,\, i=1,2$ and setting $x_3=-\dfrac{\beta_2}{\alpha_2}t$ we get 
 $$\displaystyle\psi_\lambda(\alpha_1,\alpha_2,\beta_1,\beta_2)=\beta_1^{\lambda_1-1}\beta_2^{\lambda_2+\lambda_3-1}|\alpha_2|^{-\lambda_3}\int_0^1(1-xt)^{\lambda_1-1}_+(1-t)^{\lambda_2-1}t^{\lambda_3-1}dt,$$
 where $x=\dfrac{\alpha_1\beta_2}{\alpha_2\beta_1}$. We compare this expression to the well-known representation of the Gauss hypergeometric function $F(a,b,c;x)$ for $|x|\less 1$ as the Euler-integral
 $$F(a,b,c;x)=\dfrac{\Gamma\left(c\right)}{\Gamma(b)\Gamma(b-c)}\int_0^1t^{b-1}(1-t)^{c-b-1}(1-xt)^{-a}dt;$$
 which converges for $\Re\left(c\right)\bg \Re(b)\bg 0$, and as function of $x$ extends holomorphically to $\Large\widetilde{\C P_1\setminus\{0,1,\infty\}}$, the universal covering of $\C P_1\setminus\{0,1,\infty\}$. For other values of $b$ and $c$ $F(a,b,c;x)$ must be treated by the analytic continuation with respect to $b$ and $c$. One gets finally \cite[p. 46]{gelfand2000}
 \begin{equation}\label{hypergeometric4}
 \psi_\lambda( \alpha_1,\alpha_2,\beta_1, \beta_2)=\dfrac{\Gamma(\lambda_2) \Gamma(\lambda_3) }{\Gamma(\lambda_2+\lambda_3)}\beta_1^{\lambda_1-1}\beta_2^{\lambda_2+\lambda_3-1}\vert \alpha_2\vert^{-\lambda_3}
 F\left(-\lambda_1+1,\lambda_3, \lambda_2+\lambda_3;\frac{\alpha_1\beta_2}{\alpha_2\beta_1}\right)
 \end{equation} 
where $\dfrac{\alpha_1\beta_2}{\alpha_2\beta_1}<1$. This gives the analytic continuation of $\psi_\lambda$ with respect to its arguments.

The transform given by \eqref{f90} can be treated in more than $4$ variables. Let us restrict for simplicity to the case of $9$ variables and consider the differential operator
$$\Omega_9=\det\left(\begin{array}{ccc}\partial_{\alpha_1} & \partial_{\alpha_2} & \partial_{\alpha_3} \\\partial_{\beta_1} & \partial_{\beta_2} & \partial_{\beta_3} \\\partial_{\gamma_1} & \partial_{\gamma_2} & \partial_{\gamma_3}\end{array}\right).$$
Then the function
\begin{equation*}
\begin{split}
&\varphi(\alpha_1,\alpha_2,\alpha_3,\beta_1,\beta_2,\beta_3,\gamma_1,\gamma_2,\gamma_3)\\
&=\displaystyle\int_{-\infty}^{+\infty}\int_{-\infty}^{+\infty}f(\alpha_1x_2+\beta_1x_3+\gamma_1,\alpha_2x_2+\beta_2x_3+\gamma_2,\alpha_3x_2+\beta_3x_3+\gamma_3,x_2,x_3)dx_2dx_3,
\end{split}
\end{equation*}
with $f\in S(\R^5)$, is in the kernel of $\Omega_9$. 
 \end{example}
 
  \subsection{Eigenvalues problem}
  We investigate in this section the eigenvalue problem for the differential operators associated to the Frobenius determinant of a finite abelian group. We recall that this differential operator can be factorized into
$$\displaystyle\Theta(G)\left((\partial_{g})_{g\in G}\right)=\prod_{\chi\in \hat{G}}\partial_{u_\chi},\quad \partial_{u_\chi}=\dfrac{\partial}{\partial u_\chi}.$$
We are looking for holomorphic solutions $U(u_\chi,\chi\in \hat{G})$ to
\begin{equation}
\label{f5400}
\displaystyle \prod_{\chi\in \hat{G}}\partial_{u_\chi}U=c U,\quad c\in \C^\times.
\end{equation}
We remark that by an elementary change of variables (replacing one $u_\chi$ by $\lambda u_\chi$), the equation \eqref{f5400} is equivalent to the equation
\begin{equation}
\label{f5401}
\displaystyle\prod_{\chi\in \hat{G}}\partial_{u_\chi}U= U.
\end{equation}
Let $n=\sharp{G}$. For ease of notation we identify $(u_\chi)_{\chi\in \hat{G}}$ with $(u_i)_{1\leqslant i\leqslant n}$. Take $f$ to be a holomorphic function on $\C^n$, depending on the variables $\zeta_1,\zeta_2,\ldots,\zeta_n$. Let $\zeta_1^{(0)},\zeta_2^{(0)},\ldots,\zeta_n^{(0)}$ be a point of $\C^n$. We set 
  \begin{equation}
  \label{sibirani1}\begin{split}&\left[f(\zeta_1,\zeta_2,\ldots,\zeta_n)\right]_r:=f(\zeta_1,\zeta_2,\ldots,\zeta_n)\\
  &-f(\zeta_1^{(0)},\zeta_2,\ldots,\zeta_n)-f(\zeta_1,\zeta_2^{(0)},\ldots,\zeta_n)
  -\ldots-f(\zeta_1,\zeta_2,\ldots,\zeta_n^{(0)})\\&+f(\zeta_1^{(0)},\zeta_2^{(0)},\ldots,\zeta_n)+f(\zeta_1^{(0)},\zeta_2,\zeta_3^{(0)},\ldots,\zeta_n)+\ldots+f(\zeta_1^{(0)},\zeta_2,\ldots,\zeta_{n-1},\zeta_n^{(0)})+\\&+f(\zeta_1,\zeta_2^{(0)},\zeta_{3}^{(0)},\zeta_4,\ldots,\zeta_{n-1},\zeta_n)+f(\zeta_1,\zeta_2^{(0)},\zeta_{3},\zeta_4^{(0)},\zeta_5,\ldots,\zeta_n)+\ldots+\\&f(\zeta_1,\zeta_2^{(0)},\zeta_{3},\ldots,\zeta_n^{(0)})+\ldots+f(\zeta_1,\zeta_2,\ldots,\zeta_{n-2}^{(0)},\zeta_{n-1},\zeta_n^{(0)})+f(\zeta_1,\zeta_2,\ldots,\zeta_{n-1}^{(0)},\zeta_n^{(0)})\\
  &+\ldots+(-1)^rf(\zeta_1^{(0)},\zeta_2^{(0)},\ldots,\zeta_r^{(0)},\ldots,\zeta_n)+\ldots+(-1)^rf(\zeta_1,\zeta_2,\ldots,\zeta_{n-r},\zeta_{n-r+1}^{(0)},\ldots,\zeta_n^{(0)}),
  \end{split}
  \end{equation}  
$\,r\leqslant n$. Let $x_1,x_2,\ldots,x_n$ be $n$ independent variables and $(x_1^{(0)},x_2^{(0)},\ldots,x_n^{(0)})$ a fixed point of $\C^n$. We denote by $\varphi_{(h)}$ a function of the $n-1$ variables $x_1,x_2,\ldots,x_{h-1},x_{h+1},\ldots,x_n$.
 
 Let $\varphi_{(1)}$, $\varphi_{(2)},\ldots,$$\varphi_{(n)}$ be the $n$ functions satisfying to all the following $n(n-1)$ conditions
 \begin{equation}
 \label{sibirani2}
 (\varphi_{(h)})_{x_l=x_l^{(0)}}=(\varphi_{(l)})_{x_h=x_h^{(0)}},\,\,h\not=l,\,h,l=1,2,\ldots,n. 
 \end{equation}
 Then the function 
 \begin{equation}
 \label{sibirani3}
 u=\displaystyle\sum_{h=1}^n\left[\varphi_{(h)}(x_1,x_2,\ldots,x_{h-1},x_{h+1},\ldots,x_n)\right]_{h-1}
 \end{equation} 
 satisfies the conditions 
 \begin{equation}
 \label{sibirani4}
 u_{x_i=x_i^{(0)}}=\varphi_{(i)}, i=1,2,\ldots,n
 \end{equation}
 and
 \begin{equation}
 \label{sibirani4000}
\dfrac{\partial^nu}{\partial x_1\partial x_2\ldots \partial x_n}=0
 \end{equation}
 because the functions $\varphi_{(h)}$, $1\leqslant h\leqslant n$ are functions of $n-1$ variables only. Thus $u$, given by equation \eqref{sibirani3}, is the solution to \eqref{sibirani4000} which satisfies the initial conditions \eqref{sibirani4}. Consider the following equation 
 \begin{equation}
 \label{sibirani5}
 \displaystyle \dfrac{\partial^nw}{\partial x_1\partial x_2\ldots\partial x_n }+w=\lambda(x_1,x_2,\ldots,x_n),
 \end{equation}
 where $\lambda$ is a holomorphic function on $\C^n$. If we set 
 $$\displaystyle w=v+\sum_{h=1}^n\left[\varphi_{(h)}\right]_{h-1}$$
 the equation \eqref{sibirani5} becomes 
 \begin{equation}
 \label{sibirani6}
 \displaystyle \dfrac{\partial^nv}{\partial x_1\partial x_2\ldots\partial x_n }+v=\mathcal G(x_1,x_2,\ldots,x_n) 
 \end{equation}
 with 
 $$\mathcal G(x_1,x_2,\ldots,x_n)=\lambda(x_1,x_2,\ldots,x_n)-\sum_{h=1}^n\left[\varphi_{(h)}\right]_{h-1}.$$
 Then in order to find the solution of equation \eqref{sibirani5} which for $x_i=x_i^{(0)}$ reduces to $\varphi_{(i)}$, $(i=1,2,\ldots,n)$, it suffices to find the solution of \eqref{sibirani6} which vanishes for $x_i=x_i^{(0)}$. We introduce the following entire function \cite{me}
 \begin{equation}
\label{f5402}
\displaystyle E_n(z)=\sum_{m\geqslant 0}(-1)^m\dfrac{\left(\dfrac{z}{n}\right)^{nm}}{(m!)^n},\quad n\geqslant2.
\end{equation}
and set 
$$\displaystyle F_n(x_1,x_2,\ldots,x_n)=E_{n}(n\sqrt[n]{x_1x_2\ldots x_n})=\sum_{m\geqslant0}\dfrac{(-1)^mx_1^mx_2^m\ldots x_n^m}{(m!)^n}.$$
One easily sees that $F_n(x_1,x_2,\ldots,x_n)$ satisfies the equation
\begin{equation}
\label{sibirani7}
 \displaystyle \dfrac{\partial^nv}{\partial x_1\partial x_2\ldots\partial x_n }+v=0
\end{equation}
and takes the value $1$ for $x_i=0,\,i=1,2,\ldots,n$. Let then
$$F_n(x_1-\alpha_1,x_2-\alpha_2,\ldots,x_n-\alpha_n)$$
be the solution to equation \eqref{sibirani7} which takes the value $1$ when $x_i=\alpha_i,\,i=1,2,\ldots,n$. Let us consider the function
\begin{equation}
\label{sibirani8}
\displaystyle
\begin{split}
v(x_1,\ldots,x_n)&=\int_{x_1^{(0)}}^{x_1}d\alpha_1\int_{x_2^{(0)}}^{x_2}d\alpha_2\\
&\ldots\int_{x_n^{(0)}}^{x_n}\mathcal G(\alpha_1,\alpha_2,\ldots,\alpha_n)F_n(x_1-\alpha_1,x_2-\alpha_2,\ldots,x_n-\alpha_n)d\alpha_n.
\end{split}
\end{equation}
We have 
\begin{equation*}
\begin{split}
\displaystyle \dfrac{\partial^nv}{\partial x_1\partial x_2\ldots\partial x_n }&=\int_{x_1^{(0)}}^{x_1}\int_{x_2^{(0)}}^{x_2}\ldots\int_{x_n^{(0)}}^{x_n}\mathcal G(\alpha_1,\alpha_2,\ldots,\alpha_n)\dfrac{\partial^nF_n}{\partial x_1\partial x_2\ldots\partial x_n }d\alpha_n\\
&+\mathcal G(x_1,x_2,\ldots,x_n)
=-v+G(x_1,x_2,\ldots,x_n)\end{split}
\end{equation*}
because $F_n$ satisfies \eqref{sibirani7} and $\left(\dfrac{\partial^jF_n}{\partial x_{i_1}\partial x_{i_2}\ldots\partial x_{i_j}}\right)_{|_{x_i=\alpha_i}}=0$, for $j\geqslant1$. Thus \eqref{sibirani8} is the solution of \eqref{sibirani6} which vanishes for 
$$x_i=x_i^{(0)},i=1,2,\ldots,n.$$
In conclusion if $v$ is the function defined by equation \eqref{sibirani8}, then the solution of the equation \eqref{sibirani5} which for $x_i=x_i^{(0)}$ reduces to $\varphi_{(i)}$ for $i=1,2,\ldots,n$ is 
$$\displaystyle u=v+\sum_{h=1}^n\left[\varphi_{(h)}\right]_{h-1}.$$

 \section{Frobenius Groups, Almost-Frobenius structures and toric structures}
 \label{unitsphere}
 \subsection{Frobenius Groups }
 Let $G$ be a finite group of order $n\geq 2$. We recall that to each such group and variables $(a_g)_{g\in G}$, we associated a matrix $M=(a_{gh^{-1}})_{g,h}$ whose determinant is the group determinant $\Theta(G)$. Let us consider the $(a_{gh^{-1}})_{g,h}$ as complex numbers and also the set $\mathbb{S}_G$ of matrices of Frobenius type for a fixed group $G$ with non-zero determinant. We call every such matrix an invertible Frobenius matrix associated to the finite group $G$. In this section we show that the set $\mathbb{S}_G$ of invertible Frobenius matrices is, when $G$ is abelian, a commutative linear algebraic group and a linear reductive group when $G$ is non-abelian.
 \begin{theorem} 
 \label{fro}
 The set $\mathbb{S}_G$ of invertible Frobenius matrices associated to a given finite group $G$ of order $n$ is a connected linear algebraic group. It is linear reductive when $G$ is non-abelian and is commutative when $G$ is abelian. Let $\mathfrak{s}_G$ be the Lie algebra of $\mathbb{S}_G$. If $G$ is non-abelian then the derived Lie algebra $\mathfrak{s}_G^{(1)}$ of $\mathfrak{s}_G$ is semi-simple.
 \end{theorem}
 \begin{proof}
 We first show that $\mathbb{S}_G$ is a linear algebraic group. For that we need to show that the product of two invertible Frobenius matrices is a Frobenius matrix and also that the inverse of a Frobenius is a Frobenius matrix. Let $(a_g)_{g\in G}$ and $(b_g)_{g\in G}$ two set of variables. Let $(x)=(a_{gh^{-1}})_{g,h}$ and $(y)=(b_{gh^{-1}})_{g,h}$ be two invertible Frobenius matrixes. The following computation is classical we give it here for the sake of completeness. We introduce $c_t$ satisfying
 \begin{equation}
 \label{f58}
 \displaystyle c_t=\sum_{(u,v):uv=t} a_uv_v\equiv\sum_va_ub_v.
 \end{equation}
 The product of the two matrices $(x)$ and $(y)$ is
 \begin{equation}
 \label{f59}
 \displaystyle(x)(y)=(z)=(z_{g,h})=(\sum_ra_{gr^{-1}}b_{rh^{-1}}).
 \end{equation}
 Set $rh^{-1}=s$. Then we get
\begin{equation}
\label{f60} 
\displaystyle z_{g,h}=\sum_sa_{gh^{-1}s^{-1}}b_s=c_{gh^{-1}}.
\end{equation}
 Hence $(z)=(c_{gh^{-1}})_{g,h}$. Let $(x)=(a_{gh^{-1}})_{g,h}$ be an invertible Frobenius matrix with determinant $D_{(x)}$ then it follows from \cite{conrad1998}, that the adjoint matrix of $(x)$ is given by $$A_{(x)}=\left(\dfrac{1}{n}\dfrac{\partial D_{(x)}}{\partial X_{hg^{-1}}}\right)_{g,h}.$$
 We set $U_w=\dfrac{1}{nD_{(x)}}\dfrac{\partial D_{(x)}}{\partial X_w}$. Then we have
 $$(x)^{-1}=U_{gh^{-1}}.$$
 This proves that $\mathbb{S}_G$ is a group. Moreover by definition of $\mathbb{S}_G$ the rows of any element of $\mathbb{S}_G$ are gotten from a fixed row by a permutation. This proves that $\mathbb{S}_G$ is a linear algebraic group. We next prove that $\mathbb{S}_G$ is abelian when $G$ is abelian. With the same notation as above, let $(x)$ and $(y)$ be two invertible Frobenius for $G$ with $G$ abelian. Then using calculations akin to those used in equations \eqref{f59} and \eqref{f60}, we get for the $(g,h)$ element of $(y)(x)$ the following formula
 \begin{equation}
 \label{f61}
 \displaystyle z^\prime_{g,h}=\sum_f a_fb_{gh^{-1}f^{-1}}.
 \end{equation} 
 Hence when $G$ is abelian, in accordance with equation \eqref{f58}, we have 
 $$z_{p,q}=z^\prime_{p,q},$$
 i.e $\mathbb{S}_G$ is abelian. Finally let us show that $\mathbb{S}_G$ is reductive. This amounts to show that its Lie algebra is reductive, because $\mathbb{S}_G$ is linear algebraic over the field of complex numbers which is algebraically closed. Firstly let us determine the Lie algebra $\mathfrak{s}_G$ of $\mathbb{S}_G$. To this end we use the process of epsilonization. Let $GL(n,\C[\varepsilon])$ be the general linear group with coefficients in $\C[\varepsilon]=\C\oplus\C\varepsilon,\,\varepsilon^2=0$ and $\mathbb{S}_G(\C[\varepsilon])$ the subgroup of $GL(n,\C[\varepsilon])$ consisting of those elements in $GL(n,\C[\varepsilon])$ which satisfy the same polynomial equations as the elements of $\mathbb{S}_G$. Then one has that
 
 $$Lie(\mathbb{S}_G)=\{A\in M_n(\C)|I_n+\varepsilon A\in\mathbb{S}_G(\C[\varepsilon])\}$$ is by definition the Lie algebra $\mathfrak{s}_G$ of $\mathbb{S}_G$. Let us enumerate the elements of $G$ by $S_1=e$, $S_2,\ldots,S_n$. The multiplication table of $G$ is given by
 \begin{equation}
\label{f62}
 \left[\begin{array}{cccccc}S_1 & S_2 & \ldots & S_k & \ldots & S_n \\S_2^{-1}S_1 & S_1 & \ldots & S_2^{-1}S_k & \ldots & S_2^{-1}S_n \\\ldots & \ldots & \ldots & \ldots & \ldots & \ldots \\\ldots & \ldots & \ldots & \ldots & \ldots & \ldots \\S_l^{-1}S_1 & S_l^{-1}S_2 & \ldots & S_l^{-1}S_k & \ldots & S_l^{-1}S_n \\\ldots& \ldots& \ldots & \ldots & \ldots & \ldots \\S_n^{-1}S_1 & S_n^{-1}S_2 & \ldots & S_n^{-1}S_k & \ldots & S_1\end{array}\right].
 \end{equation}
 This multiplication table serves as bookkeeping in order to define the transformations \eqref{f63}. We let the group $\mathbb{S}_G$ act on $\C^n$ with coordinates $(x_{S_1},x_{S_2},\ldots,x_{S_n})$ via the following transformations
 \begin{equation}
 \label{f63}
 \displaystyle
 \begin{split}
 x_{S_1}^\prime&=Y_{S_1}x_{S_1}+Y_{S_2}x_{S_2}+\ldots+Y_{S_k}x_{S_k}+\ldots+Y_{S_n}x_{S_n}\\
 x_{S_2}^\prime&=Y_{S_2^{-1}S_1}x_{S_1}+Y_{S_1}x_{S_2}+\ldots+Y_{S_2^{-1}S_k}x_{S_k}+\ldots+Y_{S_2^{-1}S_n}x_{S_n}\\
& \ldots\qquad\ldots\qquad\ldots\qquad\ldots\qquad\ldots\qquad\ldots\qquad\ldots\\
 x_{S_l}^\prime&=Y_{S_l^{-1}S_1}x_{S_1}+Y_{S^{-1}_lS_2}x_{S_2}+\ldots+Y_{S_l^{-1}S_k}x_{S_k}+\ldots+Y_{S_l^{-1}S_n}x_{S_n}\\
 & \ldots\qquad\ldots\qquad\ldots\qquad\ldots\qquad\ldots\qquad\ldots\qquad\ldots\\
 x_{S_n}^\prime&=Y_{S_n^{-1}S_1}x_{S_1}+Y_{S^{-1}_nS_2}x_{S_2}+\ldots+Y_{S_n^{-1}S_k}x_{S_k}+\ldots+Y_{S_1}x_{S_n}
 \end{split}
 \end{equation} 
where $(Y_{S_j^{-1}S_l})_{j,l}$ is an invertible Frobenius matrix. The infinitesimal transformations which correspond to the Lie algebra $\mathfrak{s}_G$ are generated by $I_n+\varepsilon A_k$ where
$$A_1=I_n$$
and
 \begin{equation*}
 \label{f64}
 A_k\left(\begin{array}{c}x_{S_1} \\x_{S_2} \\\vdots \\x_{S_n}\end{array}\right)=\left(\begin{array}{c}x_{k_1} \\x_{k_2} \\\vdots \\x_{k_n}\end{array}\right),\,k= 2,\ldots,n.
 \end{equation*}
More precisely one defines $A_k$ for $k=2,\ldots,n$ as follows. One sets in \eqref{f63}
$$Y_{S_1}=1,Y_{S_2}=Y_{S_3}=\ldots=Y_{S_{k-1}}=Y_{S_{k+1}}=\ldots=Y_{S_n}=0,Y_{S_k}=\varepsilon$$
and gets $I_n+\varepsilon A_k$ for $k= 2,\ldots,n$, with
 \begin{equation*}
 \label{f64}
 A_k\left(\begin{array}{c}x_{S_1} \\x_{S_2} \\\vdots \\x_{S_n}\end{array}\right)=\left(\begin{array}{c}x_{k_1} \\x_{k_2} \\\vdots \\x_{k_n}\end{array}\right),\,k= 2,\ldots,n.
 \end{equation*}
Remark that $A_k$ is a permutation matrix and $x_{k_1}=x_k$. We denote $x_{k_p}$ the coefficient of $\varepsilon$ in $x^\prime_{S_p}$. Thus the Lie algebra $\mathfrak{s}_G$ is generated as a matrix algebra by $I_n$, $(A_k)_{k=2,\ldots,n}$. In terms of vector fields we have the following generators of $\mathfrak{s}_G$. 
 \begin{equation}
 \label{f65}
 \begin{split}
 X_{S_1}&=x_{S_1}\dfrac{\partial}{\partial x_{S_1}}+x_{S_2}\dfrac{\partial}{\partial x_{S_2}}+\ldots+x_{S_n}\dfrac{\partial}{\partial x_{S_n}}\\
 X_{S_k}&=x_{k_1}\dfrac{\partial}{\partial x_{S_1}}+x_{k_2}\dfrac{\partial}{\partial x_{S_2}}+\ldots+x_{k_n}\dfrac{\partial}{\partial x_{S_n}},\,k=2,\ldots,n.
 \end{split}
 \end{equation}
 These vector fields act on the coordinate ring $$\C\left[\mathbb{S}_G\right]=\C\left[\left(x_g\right)_{g\in G},y\right]/\left(\det(x_{gh^{-1}})y-1\right).$$
 Since 
 $$x_{S_p}^\prime=Y_{S_p^{-1}S_1}x_{S_1}+Y_{S^{-1}_pS_2}x_{S_2}+\ldots+Y_{S_p^{-1}S_k}x_{S_k}+\ldots+Y_{S_p^{-1}S_n}x_{S_n}$$
 it follows that $x_{K_p}=x_{S_q}$ if $S_p^{-1}S_q=S_k$ or again if $S_q=S_pS_k$. This suggests the notation
$$x_{k_p}=x_{S_q}=x_{S_pS_k}=x_{pk},\,$$
for $k\geqslant2$ and $S_p^{-1}S_q=S_k$. In a similar fashion we set $x_{S_j}=x_j$, $j=1,2,\ldots,n$ in the vector field $X_{S_1}$. Finally we denote $X_{S_k}$ simply by $X_k$ and $X_{S_kS_j}$ by $X_{kj}$, thus we have
 \begin{equation}
 \label{f66}
 \begin{split}
\displaystyle X_{1}&=x_1\dfrac{\partial}{\partial x_1}+x_2\dfrac{\partial}{\partial x_2}+\ldots+x_n\dfrac{\partial}{\partial x_n}\\
X_k&=x_{1k}\dfrac{\partial}{\partial x_1}+x_{2k}\dfrac{\partial}{\partial x_2}+\ldots+x_{nk}\dfrac{\partial}{\partial x_n},\,k\geq 2.
 \end{split}
 \end{equation}
 Hence the Lie-bracket of any two $X_k$, $X_l$ is given by
 \begin{equation}
 \label{f67}
 [X_k,X_l]=X_{lk}-X_{kl}.
 \end{equation}
 Indeed the coefficient of $x_p$ in $X_l$ is $\dfrac{\partial}{\partial x_q}$ if
 $$p=ql\iff q=pl^{-1}$$
 and therefore the only term with a first order differential coefficient in $x_{pk}\dfrac{\partial}{\partial x_p}X_l$is 
 $$x_{pk}\dfrac{\partial}{\partial x_{pl^{-1}}}.$$
 Hence
 $$\left[X_k,X_l\right]=\sum_{p=1}^n\left(x_{pk}\dfrac{\partial}{\partial x_{pl^{-1}}}-x_{pl}\dfrac{\partial}{\partial x_{pk^{-1}}}\right).$$
 If we set $pl^{-1}=t$ and define $t$ by $pk=tx$ then $x=t^{-1}pk=lk.$ Hence 
 $$\sum_{p=1}^nx_{pk}\dfrac{\partial}{\partial x_{pl^{-1}}}=\sum_{t=1}^nx_{tlk}\dfrac{\partial}{\partial x_t}=X_{lk}.$$
 Similarly
 $$\sum_{p=1}^nx_{pl}\dfrac{\partial}{\partial x_{pk^{-1}}}=X_{kl}.$$
 Hence $[X_k,X_l]=X_{lk}-X_{kl}$.

 Let us determine the center $\mathfrak{h}$ of $\mathfrak{s}_G$. To this end we consider $Y=\displaystyle\sum_1^ne_sX_s\in\mathfrak{s}_G$; then $Y$ commutes with every element of $\mathfrak{s}_g$ if and only if it commutes with every $X_k,\,k=1,\ldots,n$. This means
 \begin{equation}
 \label{f68}
 \begin{split}
 \displaystyle[X_t,Y]&=\sum_1^ne_s[X_t,X_s]\\
&=\sum_1^ne_s(X_{st}-X_{ts})\\
&=\sum_1^n(e_{pt^{-1}}-e_{t^{-1}p})X_p\\
&=0.
\end{split}
 \end{equation}
 Hence, for every value of $p$ and $t$,
 $$e_{pt^{-1}}-e_{t^{-1}p}=0$$
 or
 \begin{equation}
 \label{f69}
 e_q-e_{t^{-1}qt}=0,
 \end{equation}
 for all values of $q$ and $t$. Therefore $Y$ belongs to the center $\mathfrak{h}$ of $\mathfrak{s}_G$ if and only if those coefficients $e_s$ whose suffixes from a conjugate class (with respect to the previously indicated inner automorphisms), are all equal. Moreover it follows from the previous analysis that the dimension of $\mathfrak{h}$ is $r$, where $r$ is the number of distinct conjugate classes in $G$.
 
We consider the derived Lie algebra $\mathfrak{s}_G^{(1)}$ of $\mathfrak{s}_G$, i.e. the ideal of $\mathfrak{s}_G$ generated by the Lie brackets $[X_k,X_l]$.
 
 If $q$ and $t^{-1}qt$ are any two conjugate operations of $G$, then 
 \begin{equation}
\label{f70}
\begin{split}
X_q-X_{t^{-1}qt}&=X_{qtt^{-1}}-X_{t^{-1}qt}\\
&=[X_{t^{-1}},X_{qt}].
\end{split}
\end{equation} 
Hence the generating operations of $\mathfrak{s}_G^{(1)}$ consist of such of the operations 
$$X_q-X_{t^{-1}qt}$$
as are linearly independent.

Suppose that $a_1,\ldots,a_s$ is a class of distinct conjugate operations of $G$ (under inner automorphisms of $G$). Then the $s-1$ operations 
$$X_{a_1}-X_{a_2},\,\,X_{a_1}-X_{a_3},\ldots,\,\,X_{a_1}-X_{a_s}$$
are linearly independent, and by a linear combination of them every operation
$$X_{a_p}-X_{a_q}\quad (p,q=1,2,\ldots,s)$$
can be expressed.

Hence each class of distinct conjugate operations of $G$ gives a set, one less in number, of linearly independent operations which must appear among the generating operations of $\mathfrak{s}_G^{(1)}$. If then $r$ still denotes the number of distinct classes of conjugate operations in $G$, $\mathfrak{s}_G^{(1)}$ is of dimension $n-r$. Moreover since the operation (which belongs to $\mathfrak{h}$)
$$X_{a_1}+X_{a_2}+\ldots+X_{a_s}$$
can not be represented as a linear combination of 
$$X_{a_1}-X_{a_2},X_{a_1}-X_{a_3},\ldots,X_{a_1}-X_{a_s},$$
we have
\begin{equation}
\label{f71}
\mathfrak{s}_G=\mathfrak{h}\oplus\mathfrak{s}_G^{(1)}.
\end{equation}
When $G$ is abelian we have seen that $\mathbb{S}_G$ is abelian hence $\mathfrak{s}_G^{(1)}=0$. But when $G$ is not abelian we show that $\mathfrak{s}_G^{(1)}$ is semi-simple. We also show that $\mathfrak{s}_G$ is linear reductive. 
We give a proof of the linear reductivity of $\mathbb{S}_G$ which, at the same time, will enable us to show the connectedness of $\mathbb{S}_G$. To this end we use the following characterization of the group algebra $\C[G]$ of $G$ given in \cite[p.~64]{serre1998}
\begin{equation}
\label{f74}
\displaystyle\C[G]\simeq \prod_{j=1}^mM_{n_j}(\C),
\end{equation} 
where $m$ is the order of the set of isomorphism classes of irreducible representations of $G$ and $n_j$ is the degree of the isomorphism class of irreducible representations. Hence the unit group of $\C[G]$ is isomorphic to $\displaystyle\prod_{j=1}^mGL(n_j,\C)$. Let us compute the matrix for left multiplication in $\C[G]$ by an element $\displaystyle\sum_{g\in G}a_gg$. Since
$$\displaystyle\left(\sum_ga_gg\right)h=\sum_ga_{gh^{-1}}g,$$
the matrix for left multiplication by $\sum a_gg$ is $(a_{gh^{-1}})$. Left multiplication by $\sum a_gg$ on the finite dimensional vector space $\C[G]$ is invertible if and only if its matrix in any basis is invertible. One sees that left multiplication by $\sum a_gg$ is invertible if and only if $\sum a_gg$ is a unit of $\C[G]$, because we have an endomorphism between finite dimensional vector spaces. Therefore the set of invertible Frobenius matrices corresponds to the unit group of the group algebra. As the latter group is isomorphic to $\displaystyle\prod_{j=1}^mGL(n_j,\C)$, it is a connected and reductive linear algebraic group. 
 We recall that linear reductivity of $\mathbb{S}_G$ means that the largest solvable ideal of $\mathfrak{s}_G$ is its center $\mathfrak{h}$. That is $\rad\mathfrak{s}_G=\mathfrak{h}$. This implies that $\mathfrak{s}_G^{(1)}$ contains no solvable ideal of $\mathfrak{s}_G$, because $\mathfrak{s}_G$ is the direct sum of $\mathfrak{s}_G^{(1)}$ and $\mathfrak{h}$, and $\rad\mathfrak{s}_G$ by construction contains all solvable ideals of $\mathfrak{s}_G$. If then $\mathfrak{k}$ is a solvable ideal of $\mathfrak{s}_G^{(1)}$, it is also a solvable ideal of the Lie algebra generated by $\mathfrak{s}_G^{(1)}$ and $\mathfrak{h}$, i.e. $\mathfrak{s}_G$ ($\mathfrak{h}$ commutes with every operation of $\mathfrak{s}_G$). And we have seen that $\mathfrak{s}_G^{(1)}$ contains no solvable ideal of $\mathfrak{s}_G$. Therefore $\mathfrak{s}_G^{(1)}$ contains no solvable ideal. Thus $\mathfrak{s}_G^{(1)}$ is semi-simple. 
 \end{proof}
 \begin{remark}
To any complex reductive group, one can associate systems of hypergeometric type called GKZ systems. Therefore we can associate GKZ systems to the Frobenius groups. See \cite{gelfand1989, kapranov1998}.
 \end{remark}
 \subsection{Almost-Frobenius structures}
 For relevant notions on Frobenius algebras and manifolds we refer to \cite{looijenga}. Our goal in this subsection is to prove that there is a structure of almost-Frobenius manifold on the complement in $\C^{n}$ of the hyperplanes defined by the vanishing of the Frobenius determinant in the case of finite abelian group $G$ of cardinal $n$. More precisely we have the
 \begin{theorem}
 \label{almost}
Let $\mathcal{H}=\{z_1=0,z_2=0,\ldots,z_n=0\}$ be the arrangement of hyperplanes determined by the coordinate axes of $\C^n$. Then the complement of the union of its hyperplanes can be endowed with a structure of an almost-Frobenius manifold.
\end{theorem}

Let $V$ be a finite dimensional complex vector space endowed with a non-degenerate complex bilinear form $g:V\times V\to \C$. Let $\mathcal{H}$ be a finite collection of linear hyperplanes in $V$ with the property that for every $H\in\mathcal{H}$ the restriction of $g$ to $H\times H$ is non-degenerate (so that $V$ is the direct sum of $H$ and its $g$-orthogonal complement $H^\bot$, a complex line in $V$) and for every $H\in \mathcal{H}$ a nonzero self-adjoint linear map $\rho_H:V\to V$ with kernel $H$. So $\rho_H$ has the form $\rho_H(v)=\alpha_h(v)\check{\alpha}_H$, where $\alpha_H\in V^\star$ has zero set $H$ and $\check{\alpha}_H\in H^\bot$.
 
 Let $V^\circ=V\setminus\cup_{H\in\mathcal{H}}H$. Every tangent space $T_pV^\circ$ can be identified with $V$ and via this identification we regard $g$ as a nondegenerate symmetric bilinear form $g$ on the holomorphic tangent bundle of $V^\circ$. The Levi-Civita connection of $g$ is the standard one on a vector space (its flat vector fields are the constant vector fields). We define a commutative product on the holomorphic tangent bundle of $V^\circ$ by 
 $$\displaystyle X\cdot Y=\sum_H\omega_H(X)\rho_H(Y),\, \omega_H=\dfrac{d\alpha_H}{\alpha_H}.$$
 For $\lambda\in\C$ we define a connection $\nabla^\lambda$ on the holomorphic tangent bundle of $V^\circ$ as usual:
 $$\displaystyle \nabla^\lambda_X(Y)=\nabla_XY+\lambda X\cdot Y.$$
 \begin{proposition}[\cite{looijenga}]
 If $\sum_H\rho_H$ is the identity transformation, then the Euler vector field $E$ on $V$, characterized by $E_p=(p,p)\in\{p\}\times V=T_pV$, serves as identity element for the above product (but is not flat for $g$); the product has $E$-degree $1$ and $g$ has $E$-degree $2$.
 If the system $(V,g,(\rho_H)_H)$ has the Dunkl property, meaning that for every linear subspace $L\subset V$ of codimension $2$ obtained as an intersection of members of $\mathcal{H}$ the sum $\sum_{H\in\mathcal{H},H\supset L}\rho_H$ commutes with each of its terms, then the product is associative, the connection $\nabla^\lambda$ is flat for every $\lambda\in \C$, and a potential function is given by 
 $$\displaystyle \Phi:=\sum_{H\in\mathcal{H}}\dfrac{g(\check{\alpha}_H,\check{\alpha}_H)}{2\alpha_H(\check{\alpha}_H)}\alpha_H^2\log\alpha_H.$$ 
 \end{proposition}
 We apply this proposition to the case of the hyperplanes given by the vanishing of the Frobenius determinant in the case of a finite abelian group of order $n$. We recall (see \theoref{dedekind}) that we have the following factorization of this determinant:
$$\displaystyle \Theta(G)(X_{g}):=\det(X_{gh^{-1}})=\prod_{\chi\in \hat{G}}(\sum_{g\in G}\chi(g)X_g).$$
Under a change variable (see Section \ref{frobenius}) one can transform it into $\prod_{\chi\in \hat{G}}z_\chi$ which we identify with $\prod_{i=1}^nz_i$.

\begin{prooof}
Let $g:\C^n\times\C^n\to \C$ such that $\displaystyle g(z,z^\prime)=\sum_{i=1}^nz_iz_i^\prime$, for $z=(z_1,z_2,\ldots,z_n)$ and $z^\prime=(z_1^\prime,z_2^\prime,\ldots,z_n^\prime)$. Let us denote by $H_i$ the hyperplane $z_i=0$. Let us first of all show that $g|_{H_i}$ is non-degenerate. For $u=(z_1,z_2,\ldots,z_{i-1},0,z_{i+1},z_{i+2},\ldots,z_n)\in H_i\setminus\{0\}$, there exists $j\in\{1,\ldots,n\}\setminus\{i\}$ such that $z_j­\neq0$. So $v=(0,0,\ldots,0,\frac{1}{z_j},0,\ldots,0)\in H_i$ (zeros everywhere except at the $j$-th position). We verify that $g(u,v)=1$ and hence $g|_{H_i}$ is non-degenerate. Set $\rho_{H_i}=z_ie_i$ where $z_i:z\mapsto z_i$ is the projection map from $\C^n\to \C$ and $e_i$ is the $i$-th element of the canonical $\C$ basis of $\C^n$. We clearly have 
$$\displaystyle \sum_{i=1}^n\rho_{H_i}=\Id.$$  
Besides $\rho_{H_i}\circ\rho_{H_j}=0,\,i­\neq j$; so the Dunkl condition is satisfied. Finally it is easy to see hat every $\rho_{H_i}$ is self-adjoint with respect to the metric $g$. Hence from the previous proposition we have a structure of almost-Frobenius structure on $V^\circ$.
\end{prooof}
\subsection{Frobenius determinants and toric varieties}
Let us start from the Frobenius determinant of a finite abelian group $G$ of order $n=\sharp{G}\geqslant3$
$$\Theta(G)(X_g)=\det(X_{gh^{-1}}).$$
We consider the associated smooth hypersurface of $\C^n$ (smoothness follows from the Euler relation) given by
$$\Theta(G)(X_g)=\det(X_{gh^{-1}})=1.$$
For example if $G=\Z/3\Z$ and we restrict to real variables we get the hypersurface 
$$\Sigma_3=\{(X,Y,Z)\in \R^3: X^3+Y^3+Z^3-3XYZ=1\}.$$
It is a surface of revolution well-known under the name hexenhut. Let us introduce the coordinates  \begin{equation*}
\begin{rcases}
&X+jY+j^2Z=x-iy\\
&X+j^2Y+jZ=x+iy\\
&X+Y+Z=z\\
&j=\exp(2i\pi/3)
\end{rcases}\Longrightarrow \Sigma_3\longleftrightarrow z(x^2+y^2)=1.
\end{equation*}
The picture of the surface $\{(x,y,z)\in \R^3: z(x^2+y^2)=1\}$ is shown below.
\begin{center}
\begin{tikzpicture}
    \begin{axis}[
        axis equal image,
        hide axis,
        z buffer = sort,
        view = {110}{23},
        scale = 1.5 ]
        \addplot3[
            surf,
            shader = faceted interp,
            samples = 25,
            domain = 0.04:6.5,
            y domain = 0:360,
            colormap name = mycolormap,
            thin
        ](
            {1/sqrt(x)*cos(y)},
            {1/sqrt(x)*sin(y)},
            {x}
        );
    \end{axis}
\end{tikzpicture}
\captionof{figure}{:\,A rotated Hexenhut: $z(x^2+y^2)=1$.}
\label{default}
\end{center}
Returning to the general case we make the change of variables (a birational morphism) as usual: setting $x_g=\displaystyle\sum_{g\in G}\chi(g)X_g$ we get an hypersurface of the form $\displaystyle\prod_{\chi\in \hat{G}}x_\chi=1$,  which we identify with $\displaystyle\prod_{i=1}^nx_i=1$. Homogenizing we obtain the projective hypersurface $\mathcal S$ defined by the following equation
\begin{equation}
\label{tor1}
\displaystyle\prod_{i=1}^nx_i=x_0^n,\,n\geqslant3.
\end{equation}
This hypersurface $\mathcal S$ will be our main object of study in this part.
\begin{lemma}
\label{singularpoints}
The variety $\mathcal S$ defined by \eqref{tor1} has exactly $n$ families of singular points and it contains exactly $n$ projective $(n-2)$-dimensional subspaces.
\end{lemma}
\begin{proof}
Let $\delta=\left[\delta_0,\delta_1,\ldots,\delta_n\right]\in\C P_n$. It is a singular point of $\mathcal S$ if and only if it belongs to $\mathcal S$: $\left[\delta_0,\delta_1,\ldots,\delta_n\right]\in \mathcal S$ and the partial derivatives of the defining polynomial of $\mathcal S$ vanish at $\delta$. This means
\begin{equation}
\label{tor2}
\begin{split}
&\displaystyle \delta_0^n=\prod_{i=1}^n\delta_i\\
&n\delta_0^{n-1}=0,\,\,\prod_{\substack{i=1\\ i\not=j}}^n\delta_i=0, j\in\{1,2,\ldots,n\}.
\end{split}
\end{equation}
This requires that $\delta_0=0$ as well as two of the $\delta_i$ for $i\in\{1,2,\ldots,n\}$ be zero. This gives the families
$$\left[0,0,0,\delta_3,\ldots,\delta_n\right],\left[0,0,\delta_2,0,\delta_4,\ldots,\delta_n\right],\ldots,\left[0,\delta_1,\delta_2,\delta_3,\ldots,\delta_{n-2},0,0\right],$$
with $\delta_i\in\C,1\leqslant i\leqslant n$ not all zero. In order to find the $(n-2)$-dimensional subspaces in $\mathcal S$ we first consider the $(n-2)$-dimensional subspaces on the intersection of $\mathcal S$ with the hyperplane $x_0=0$ of $\C P_n$. The intersection is given by the equation
$$0=\displaystyle\prod_{i=1}^nx_i$$
and therefore consists of exactly the hyperplanes $h_i:=\{x_0=0=x_i=0\},1\leqslant i\leqslant n$. Any hyperplane of $\mathcal S$ which is not contained inside $\mathcal S\cap\{x_0=0\}$ gives a hyperplane of the affine variety given by the equation
$$1=\displaystyle\prod_{i=1}^nx_i.$$
Inserting the equation of an $(n-2)$-dimensional subspace $h$ with directions $\vec{v}_1=(v_1^j)_{1\leqslant j\leqslant n}$, $\vec{v}_2=(v_2^j)_{1\leqslant j\leqslant n},$ $\ldots$, $\vec{v}_{n-2}=(v_{n-2}^j)_{1\leqslant j\leqslant n}$ passing through the point $p=(p_1,\ldots,p_n)$ we get 
\begin{equation*}
\begin{split}
&1=(v_1^1\lambda_1+v_2^1\lambda_2+\ldots+v_{n-2}^1\lambda_{n-2}+p_1)(v_1^2\lambda_1+v_2^2\lambda_2+\ldots+v_{n-2}^2\lambda_{n-2}+p_2)\\
&\ldots (v_1^n\lambda_1+v_2^n\lambda_2+\ldots+v_{n-2}^n\lambda_{n-2}+p_n),\,\lambda_1,\ldots,\lambda_{n-2}\in \C.
\end{split}
\end{equation*}
By expanding we see that
$$v_1^1v_1^2\ldots v_1^n=0.$$
So one of the $v_1^j=0,$ $j=1,\ldots,n$. Suppose for instance that $v_1^1=0$. Then $p_1v_1^2v_1^3\ldots v_1^n=0$. $p_1\not=0$ because (by construction) $p_1p_2\ldots p_n=1$. Then $v_1^2v_1^3\ldots v_1^n=0$. Assume that $v_1^2=0$ then $p_1p_2v_1^3v_1^4\ldots v_1^n=0$. Continuing in this fashion, we see that $v_1^j=0$, $j=1,\ldots,n$. Likewise expanding the remaining terms we see that
$$v_i^j=0,i\in\{1,\ldots,n-2\},j\in\{1,\ldots,n\}.$$
This contradicts the fact that $\vec{v}_1,\ldots,\vec{v}_{n-2}$ are directions for the hyperplane $h$. Thus there are exactly $n$-hyperplanes $h_i$ which lie in $\mathcal S$.
\end{proof}
\begin{remark}
This lemma should be compared to the case of smooth cubic surfaces in $\C P_3$ which contain exactly $27$ lines, see \cite[p.~172]{mumford1976}.
\end{remark}
In this final part we wish to prove the following assertion
\begin{theorem}
Let $n\geqslant3$. The projective variety $\mathcal S=\proj(\C\left[x_0,x_1,\ldots,x_n\right]/\displaystyle \langle(x_0^n-\prod_{i=1}^nx_i)\rangle)$ is a normal toric variety.
\end{theorem}
We introduce firstly some definitions about toric geometry. Our main reference is \cite{cox2011}.
\begin{definition}
A semigroup $\mathscr S$ is a set which is stable under an associative addition operation (commutative) and which contains a unit element denoted $0$.
\end{definition}
\begin{remark}
Given a semigroup $\mathscr S$ the semigroup algebra $\C[\mathscr S]$ is the algebra with basis $\chi^u$ for $u\in \mathscr S$ together with the rules $\chi^0=1$, $\chi^{u_1}\chi^{u_2}=\chi^{u_1+u_2}$, $u_1,u_2\in \mathscr S$, extended by $\C$-linearity. A morphism $\phi:\mathscr S\to \mathscr S^\prime$ between two semigroups $\mathscr S$ and $\mathscr S^\prime$ is an additive morphism which maps zero to zero. A morphism a semigroup $\phi:\mathscr S\to \mathscr S^\prime$ induces a morphism of $\C$-algebra $f:\C\left[\mathscr S\right]\to\C\left[\mathscr S^\prime\right]$ given by $\chi^u\mapsto \chi^{\phi(u)}$. We will generally assume that $\mathscr S$ is integral, i.e. we can embed $\mathscr S$ as a semigroup in a some lattice. This implies that $\C\left[\mathscr S\right]$ is a domain. In general an integral finitely generated semigroup $\mathscr S$ is given by the image of a morphism of semigroups $\phi:\N^n\to \Z^m$. Such a morphism induces a surjective morphism $f:\C\left[t_1,\ldots,t_n\right]\to\C\left[\mathscr S\right]$. And the kernel of $f$ is the ideal $I:=\langle t^a-t^b,\,a,b\in\N^n:\phi(a)=\phi(b)\rangle$.
\end{remark}
\begin{definition}
\label{tor271}
Let $k$ be a field. A $k$-torus of dimension $n$ is an algebraic group over $k$ which is isomorphic (as an algebraic group) to the $n$-dimensional algebraic torus $G_{m,k}^n:=\spec(k\left[t_1^{\pm},\ldots,t_n^{\pm}\right])$
\end{definition}
\begin{definition}
\label{tor270}
A toric variety $X$ over $\C$ is a non necessarily normal, irreducible $\C$-algebraic variety which contains an algebraic torus $G_{m,\C}^n$, for some positive $n$ as a Zariski open (automatically dense) subset and such that the action of the torus on itself extends to an algebraic action of $G_{m,\C}^n$ on $X$. (By algebraic action we mean an action $G_{m,\C}^n\times X\to X$ given by a morphism).  
\end{definition}
\begin{remark}
\label{tor272}
One way to view a $k$-scheme $X$ for $k$ a field of characteristic zero (say) is to view it as its functor of points (\cite[p.~253]{eisenbud1999}), associating with any $k$-algebra $R$ the set $h_X(R)=\Hom_R(\spec(R),X)$ and with any morphism $f:R\to S$ of $k$-algebras, the map $h_X(f):X(R)\to X(S)$ induced by the composition with the morphism of schemes $f^{\star}:\spec(S)\to\spec(R)$. In this respect the functor of points of $X=G^n_{m,k}$ is simply given by $h_X(R)=(R^{\times})^n$, and for a morphism $f:R\to S$, by the map $h_X(f):(R^{\times})^n\to(S^{\times})^n$ which is induced by $f$ coordinate-wise. That $X$ be an algebraic group means that these sets $h_X(R)$ are endowed (by coordinate-wise multiplication) with a structure of groups, and that the maps $h_X(f)$ are morphisms of groups. The unit element is the point $(1,1,\ldots,1)$ of $X(k)$.

Likewise one can describe the functor of points of $P^n_k$ for a field $k$. For a $k$-algebra $R$, $P^n_k(R)$ is given by (notation) a pair $(\mathscr{L}, \left[s_0,s_1,\ldots,s_n\right])$ where $\mathscr{L}$ is a line bundle (rank one locally free $\mathscr{O}_{\spec(R)}$-module) on $\spec(R)$ and the $s_i$ are global sections of $\mathscr{L}$ such that $\spec(R)=\displaystyle\cup_{i=0}^nD(s_i)$, where $D(s_i):=\{x\in \spec(R), s_i(x)\not =0\}$, modulo isomorphism; here $s_i(x)$ denotes the image of $s_i$ in the fiber of $\mathscr L$ over $x$. In this respect $R$-points of a projective subscheme $\mathscr Z$ of $P^n_k$ correspond to $R$-points of $P^n_k$ which "satisfy" the defining equations of $\mathscr Z$, \cite[p.~260]{eisenbud1999}.

Finally the action of an algebraic group variety $G$ over $k$ on a variety $X$ is given by a functorial (in $R$) action of the group $G(R)$ on $X(R)$. 
\end{remark}
\begin{proooof}
We set $\mathcal S=\proj(\C\left[x_0,x_1,\ldots,x_n\right]/\displaystyle \langle(x_0^n-\prod_{i=1}^nx_i)\rangle)$. Since $\langle\displaystyle (x_0^n-\prod_{i=1}^nx_i)\rangle$ is a homogeneous prime ideal in $\C\left[x_0,\ldots,x_n\right]$, we have a closed immersion $\mathcal S\hookrightarrow Y$. Thus $\mathcal S$ is a projective subscheme of $Y$, which is separated over $Y$ (\cite[p.~99]{hartshorne1977}). As $Y:=\C P_n$ is separated over $\C$, $\mathcal S$ is separated over $\C$ (composition of separated morphisms is separated, \cite[p.~99]{hartshorne1977}). Moreover because irreducible closed sets of $Y$ correspond to homogeneous prime ideals of $\C\left[x_0,x_1,\ldots,x_n\right]$, $\mathcal S$ is irreducible. To show that $\mathcal S$ is an integral scheme it remains to show that $\mathcal S$ is reduced. For that it suffices to show that $\mathcal S$ can be covered by open affines whose coordinate rings are reduced. Let $D_+(x_0)$,$\ldots$,$D_+(x_n)$ be the open cover of $Y=\proj(\C\left[x_0,x_1,\ldots,x_n\right])$ by principal open sets. We take as open covering of $\mathcal S$ the following
$$\displaystyle\bigcup_{i=0}^n(D_+(x_i)\cap \mathcal S).$$
We have 
$$\displaystyle D_+(x_0)\cap \mathcal S=\spec\left(\dfrac{\C\left[\frac{x_1}{x_0},\ldots,\frac{x_n}{x_0}\right]}{\langle1-\frac{x_1}{x_0}\frac{x_2}{x_0}\ldots\frac{x_n}{x_0}\rangle}\right)=D_+(x_1)\cap D_+(x_2)\cap\ldots\cap D_+(x_n)\cap \mathcal S.$$ 
$D_+(x_0)\cap \mathcal S$ is a torus. Indeed if we set $y_j=\dfrac{x_j}{x_0},\,1\leqslant j\leqslant n-1$ then we find 
$$D_+(x_0)\cap \mathcal S=\spec(\C\left[y_1,y_1^{-1},y_2,y_2^{-1},\ldots,y_{n-1},y_{n-1}^{-1}\right])$$
and this is an $(n-1)$-dimensional torus by definition.

For $1\leqslant j\leqslant n$
$$D_+(x_j)\cap \mathcal S=\spec\left(\dfrac{\C\left[\frac{x_0}{x_j},\frac{x_1}{x_j},\ldots,\frac{x_{j-1}}{x_j},\frac{x_{j+1}}{x_j},\ldots,\frac{x_n}{x_j}\right]}{\langle(\frac{x_0}{x_j})^n-(\prod^n_{\substack{i=1\\ i\not=j}}x_i)/x_j^{n-1}\rangle}\right).$$
If we set $z_0=\frac{x_0}{x_j}$, $z_1=\frac{x_1}{x_j}$,$\ldots$,$z_{j-1}=\frac{x_{j-1}}{x_j}$, $z_{j+1}=\frac{x_{j+1}}{x_j}$,$\ldots$,$z_n=\frac{x_n}{x_j}$ we find
$$D_+(x_j)\cap \mathcal S=\spec\left(\C\left[z_0,z_1,\ldots,z_{j-1},z_{j+1},\ldots,z_n\right]/ \langle z_0^n-\prod^n_{\substack{i=1\\ i\not=j}}z_i\rangle\right).$$
Let us exhibit the coordinate ring of each $D_+(x_j)\cap \mathcal S$, $i=0,1,\ldots,n$ as the subring of some integral ring. This will imply the fact that they are reduced.

For $i=0$ the coordinate ring of $D_+(x_0)\cap \mathcal S$ is given by $\C\left[y_1,y_1^{-1},y_2,y_2^{-1},\ldots,y_{n-1},y_{n-1}^{-1}\right]$ and this is an integral ring hence reduced. Let $e_0=(1,0,\ldots,0)$, $e_1=(0,1,0,\ldots,0)$, $e_2=(0,0,1,\ldots,0)$,$\ldots$, $e_{j-1}=(0,0,\ldots,1,0,\ldots,0)$ (1 at the $j$ entry), $e_{j}=(0,0,\ldots,1,\ldots,0)$, $\ldots$, $e_n=(0,0,\ldots,0,1)$ the canonical basis of $\N^{n+1}$ as a semigroup. For each $1\leqslant j\leqslant n-1$ the semigroup generated by $e_0=(1,0,\ldots,0)$, $e_1=(0,1,0,\ldots,0)$, $e_2=(0,0,1,\ldots,0)$,$\ldots$, $e_{j-1}=(0,0,\ldots,1,0,\ldots,0)$ (1 at the $j$ entry), $e_{j+1}$, $\ldots$, $e_n=(0,0,\ldots,0,1)$ is isomorphic to $\N^{n}$. Let also $(e^{\prime}_k)_{1\leqslant k\leqslant n-1}$ be the canonical basis of the lattice $\Z^{n-1}$ as a $\Z$-module. If $1\leqslant j\leqslant n-1$ we define the morphism of semigroups $\phi_j$ by
\begin{equation*}
\begin{split}
&\phi_j(e_0)=-e_j^\prime,\phi_j(e_1)=e_1^\prime-e_j^\prime,\ldots,\phi(e_{j-1})=e_{j-1}^\prime-e_j^\prime,\\
&\phi_j(e_{j+1})=e_{j+1}^\prime-e^\prime_{j},\ldots,\phi_j(e_n)=-2e_j^\prime-(e_1^\prime+e_2^\prime+\ldots+e_{j-1}^\prime+e_{j+1}^\prime+\ldots+e^\prime_{n-1}).
\end{split}
\end{equation*}
For every $j$ the image of $\phi_j$ is a subsemigroup $\mathscr S_j$ of $\Z^{n-1}$. Also $\phi_j$ induces a morphism of $\C$-algebras $f_j:\C\left[\N^n\right]=\C\left[z_0,z_1,\ldots,z_{j-1},z_{j+1},\ldots,z_n\right]\displaystyle\to\C\left[\Z^{n-1}\right]$ which is surjective onto $\C\left[\mathscr S_j\right]$. The kernel of $f_j$ is generated by the polynomial $(z_0^n-\prod_{\substack{i=1\\i\not=j}}^nz_i)$. This gives the morphisms
$$\C\left[\Z^{n-1}\right]\hookleftarrow\C\left[\mathscr S_j\right]\simeq\C\left[z_0,z_1,\ldots,z_{j-1},z_{j+1},\ldots,z_n\right]/\langle(z_0^n-\prod_{\substack{i=1\\i\not=j}}^nz_i)\rangle.$$
Since $\C\left[\Z^{n-1}\right]$ is integral, $\C\left[\mathscr S_j\right]$ is integral hence reduced, for any $1\leqslant j\leqslant n-1$.

Let $j=n$ we define the morphism of semigroup with image $\mathscr S_n$ by
\begin{equation*}
\begin{split}
&\phi_n(e_0)=e_1^\prime+e_2^\prime+\ldots+e_{n-1}^{\prime},\phi_n(e_1)=2e_1^\prime+e_2^\prime+\ldots+e_{n-1}^\prime,\ldots,\\
&\phi_n(e_{l})=e_{1}^\prime+e_2^\prime+\ldots+e_{l-1}^\prime+2e_l^\prime+e_{l+1}^\prime+\ldots+e_{n-1}^\prime,\ldots,\\
&\phi_n(e_{n-1})=e_1^\prime+e_2^\prime+\ldots+e_{n-2}^\prime+2e^\prime_{n-1}.
\end{split}
\end{equation*}
We then have an induced morphism of semigroup algebras $$f_n:\C\left[\N^n\right]=\C\left[z_0,z_1,\ldots,z_{j},z_{j+1},\ldots,z_{n-1}\right]\to\C\left[\Z^{n-1}\right]$$ with kernel given by the ideal $\langle (z_0^n-\prod_{i=1}^{n-1}z_i)\rangle$ and which is surjective onto $\C\left[\mathscr S_n\right]$. Thus
$$\C\left[\Z^{n-1}\right]\hookleftarrow\C\left[\mathscr S_n\right]\simeq\C\left[z_0,z_1,\ldots,z_{j},z_{j+1},\ldots,z_{n-1}\right]/\langle(z_0^n-\prod_{i=1}^{n-1}z_i)\rangle.$$ 
Therefore $\C\left[\mathscr S_n\right]$ is integral hence reduced. 

Let us show that $\mathcal S$ is normal. To see this we need to show (for instance) that the algebra $\C[\mathscr S_0]=\C\left[y_1,y_1^{-1},y_2,y_2^{-1},\ldots,y_{n-1},y_{n-1}^{-1}\right]$ as well as $\C\left[\mathscr S_j\right]$, $1\leqslant j\leqslant n$ are normal rings. $\C\left[y_1,y_1^{-1},y_2,y_2^{-1},\ldots,y_{n-1},y_{n-1}^{-1}\right]$ is as is well-known normal: it is the localization of the UFD (which is normal) $\C\left[y_1,\ldots,y_{n-1}\right]$ at $y_{1}\ldots y_{n-1}$ (see \cite[p.~38]{cox2011}). Since $D_+(x_0)\cap \mathcal S$ is non-empty and has a normal coordinate ring $\C[\mathscr S_0]$, $D_+(x_0)\cap \mathcal S$ is a normal affine variety. Besides $D_+(x_0)\cap \mathcal S\subset D_+(x_i)\cap \mathcal S$, $1\leqslant i\leqslant n$. Hence each $D_+(x_i)\cap \mathcal S$, $1\leqslant i\leqslant n$ is non-empty. Finally using \cite[(A,31), p. 175; p. 128]{dieudonne} one sees that the coordinate rings $\C[\mathscr S_i]$, $1\leqslant i\leqslant n$ of the $D_+(x_i)\cap \mathcal S$, $1\leqslant i\leqslant n$ are  normal rings. Hence all the coordinate rings $\C[\mathscr S_i]$, $0\leqslant i\leqslant n$ of the $D_+(x_i)\cap \mathcal S$, $1\leqslant i\leqslant n$ are normal rings. Thus $\mathcal S$ is covered by non-empty affine patches $D_+(x_i)\cap \mathcal S$, $0\leqslant i\leqslant n$ whose coordinate rings are normal rings. Hence $\mathcal S$ is a normal.

In conclusion the scheme $\mathcal S$ is an integral normal separated projective variety of $\C$. To finish the proof we must show that it admits a torus action. In order to see this we remark (as seen above) that 
$$\displaystyle D_+(x_0)\cap \mathcal S=\spec\left(\dfrac{\C\left[\frac{x_1}{x_0},\ldots,\frac{x_n}{x_0}\right]}{\langle1-\frac{x_1}{x_0}\frac{x_2}{x_0}\ldots\frac{x_n}{x_0}\rangle}\right)=D_+(x_1)\cap D_+(x_2)\cap\ldots\cap D_+(x_n)\cap \mathcal S$$ 
is openly immersed in $\mathcal S$ as a non-empty $(n-1)$-dimensional torus. Since $\mathcal S$ is irreducible, $D_+(x_0)\cap \mathcal S$ is necessarily dense in $\mathcal S$. Finally the action of $G_{m,\C}^{n-1}$ on $\mathcal S$ which restricts to an action of $G_{m,\C}^{n-1}$ on itself is given on the set of $R$-points (the involved line bundle remains fixed) for a $\C$-algebra $R$ by
$$(t_1,\ldots,t_{n-1})\cdot\left[s_0,\ldots,s_n\right]=\left[s_0,t_1s_1,t_2s_2,\ldots,t_{n-1}s_{n-1},t_1^{-1}t_2^{-1}\ldots t_{n-1}^{-1}s_n\right]$$
where $(t_1,\ldots,t_{n-1})\in (R^\times)^{n-1}$ and $\left[s_0,\ldots,s_n\right]$ is as above.
\end{proooof}
\section{Data Availability}
This paper is not using any external data.
 
\end{document}